\documentclass[0pt,reqno]{amsart}
\usepackage{pgf,tikz}
\usepackage{mathrsfs}
\usepackage{amsrefs}
\usetikzlibrary{arrows}
\usepackage{amssymb}
\usepackage{graphicx}
\usepackage{amsmath}
\usepackage{centernot}
\usepackage{blindtext}
\usepackage[english]{babel}
\usepackage[utf8]{inputenc}
\usepackage{pgf,tikz}
\usepackage{mathrsfs}
\usetikzlibrary{arrows}
\usepackage{amssymb}
\usepackage{graphicx}
\usepackage{amsmath}
\usepackage{centernot}
\usepackage{blindtext}
\usepackage{enumitem}
\usepackage{xcolor}
\usepackage{thmtools}
\usepackage{bbm}
\usepackage{thm-restate}
\usepackage{mathtools}
\usepackage{bbold}

\usepackage{hyperref}

\usepackage{cleveref}

\usepackage{ wasysym }
\usepackage{mathtools}
\usepackage{physics}
\numberwithin{equation}{section}
\usepackage{amsthm}
\usepackage{tikz-cd} 
\usepackage{mathtools}

\DeclarePairedDelimiter\floor{\lfloor}{\rfloor}

\newcommand{\Id}{\text{Id}}

\newcommand{\R}{\mathbb{R}}

\newcommand{\N}{\mathbb{N}}

\newcommand{\C}{\mathbb{C}}

\newcommand{\Z}{\mathbb{Z}}

\newtheorem{theorem}{Theorem}[section]

\newtheorem{corollary}[theorem]{Corollary}
\newtheorem{Claim}[theorem]{Claim}

\newtheorem{Lemma}[theorem]{Lemma}
\newtheorem{Definition}[theorem]{Definition}
\newtheorem{Proposition}[theorem]{Proposition}
\usepackage[a4paper, total={6in, 9in}]{geometry}
\usepackage[utf8]{inputenc}
\usepackage[normalem]{ulem}

\title{The Return Map of the Cross Section of Horizontally Short Lattice Surfaces is Weakly Mixing}
\author{Albert Artiles}

\date{\today}
\begin{document}

\maketitle

\begin{abstract}
\noindent We prove that the return map of the unstable horocycle flow on the space of horizontally short translation surfaces associated to a lattice surface 
$(X,\omega)$ is weakly mixing. This extends the result of Cheung–Quas for the square torus to all lattice surfaces. The proof adapts their criterion for weak mixing and uses quantitative bounds for Siegel–Veech transforms restricted to the Poincaré section of horizontally short surfaces.
\end{abstract}

\section{Introduction}

The classical Boca--Cobeli--Zaharescu (BCZ) map is the transformation $T:\Omega \to \Omega$ defined by
\[
T(x,y) = \bigl(y, -x + \lfloor \tfrac{1+x}{y} \rfloor y \bigr),
\]
where $\Omega$ is the region of the plane given by $0<x\leq 1$, $0<y\leq 1$, and $x+y>1$. This map was introduced by Boca-Cobeli-Zaharescu in \cite{BocCobZah} to study asymptotic properties related to the distribution of $h$-spacings of Farey fractions. Independently, Athreya and Cheung \cite{AthChe} rediscovered the map while analyzing the dynamics of the unstable horocycle flow the space of unimodular lattices of $\R^2$.

A \emph{horizontally short lattice} is a unimodular lattice $g\mathbb{Z}^2$ with $g \in SL(2,\mathbb{R})$ such that $g\mathbb{Z}^2 \cap \{(x,0)\in \R^2 : 0 < x \leq 1\} \neq \varnothing$. Athreya-Cheung identified the BCZ map as the return map of the unstable horocycle flow on the set of horizontally short lattices. Subsequently, Athreya--Chaika--Leli\`evre \cite{AthChaLel}, Uyanik--Work \cite{UyaWor}, Taha \cite{Tah}, and Sanchez \cite{San} extended this construction to the setting of translation surfaces and their set of holonomy vectors.

Let $(X,\omega)$ be a lattice surface with Veech group $\Gamma$ and a horizontal saddle connection of length 1 and no shorter horizontal saddle connections. In this paper, we assume that $SL(2,\R)/\Gamma$ has only one cusp and that $-\Id\in \Gamma$, to simplify the computations, but all the results still hold as long as $SL(2,\R)/\Gamma$ has finitely many cusps. Let $\Lambda^*$ be the collection of holonomy vectors of $(X,\omega)$ and define $\Lambda$ to be the set of visible elements of $\Lambda^*$. We define the set of horizontally short translation surface in the $SL(2,\R)$-orbit of $(X,\omega)$ as $\mathcal{L}=\{g\Gamma\in SL(2,\R)/\Gamma: g\Lambda\cap\{(x,0)\in\R^2: 0<x\leq 1\}\neq \varnothing\}.$

Athreya \cite{Ath} showed that $\mathcal{L}$ is a Poincaré section for the unstable horocycle flow on $SL(2,\R)/\Gamma$. Let $T=T_\mathcal{L}:\mathcal{L}\rightarrow\mathcal{L}$ denote the return map under the unstable horocycle flow and $m$ denote the induced probability measure on $\mathcal{L}$. Our main theorem is the following:

\begin{theorem}\label{Thm: Main}
The dynamical system $(\mathcal{L}, T, m)$ is weakly mixing.
\end{theorem}

%Throughout this paper, we let $\Gamma$ denote a non-uniform lattice in $SL(2,\mathbb{R})$ with a single cusp. (This assumption is made for convenience; all results remain valid when $\Gamma$ has finitely many cusps.)

%Let $V = \{v_1, v_2, \ldots, v_m\}$ be a finite collection of nonzero vectors in $\mathbb{R}^2$ such that each $\Gamma v_i$ is discrete and, for $i \neq j$, the sets of directions of $\Gamma v_i$ and $\Gamma v_j$ are disjoint. For convenience, we take $v_1 = (0,1)^{\mathsf{T}}$. Define
%\[
%\Lambda = \bigsqcup_{i=1}^m \Gamma v_i
%\quad \text{and} \quad
%\mathcal{L} = \{ g\Gamma : g\Lambda \cap (0,1] \neq \varnothing \}.
%\]
%Athreya \cite{Ath} showed that $\mathcal{L}$ is a Poincar\'e section for the unstable horocycle flow on $SL(2,\mathbb{R})/\Gamma$. Let $T_{V,\Gamma}$ denote the corresponding return map on $\mathcal{L}$, and let $m$ be the measure on $\mathcal{L}$ induced by the probability Haar measure on $SL(2,\mathbb{R})/\Gamma$. Our main result is the following.

%\begin{theorem}\label{Thm: Main}
%The dynamical system $(\mathcal{L}, T_{V,\Gamma}, m)$ is weakly mixing.
%\end{theorem}

This result generalizes the paper by Cheung-Quas \cite{CheQua} where they prove the same result in the case when $(X,\omega)$ is the square torus of unit area.

Since $(0,1)^{\mathsf{T}} \in \Lambda$, there exists $\alpha > 0$ such that $\begin{bmatrix} 1 & \alpha \\ 0 & 1 \end{bmatrix}$
generates the stabilizer of $(1,0)^{\mathsf{T}}$ in $\Gamma$. Athreya-Chaika-Leli\`evre \cite{AthChaLel} showed that $\mathcal{L}$ can be parameterized by the set
\[
\Omega = \{ (s,t) \in \mathbb{R}^2 : 0 < s \leq 1, \ 1 - \alpha s < t \leq 1 \},
\]
via the identification
\[
(s,t) \longleftrightarrow
\begin{bmatrix} s & t \\ 0 & s^{-1} \end{bmatrix}\Gamma.
\]
Uyanik-Work \cite{UyaWor} computed a parameterization for $\mathcal{L}$ when $SL(2,\R)/\Gamma$ has finitely many cusps.

Under this parameterization, the measure $m$ corresponds to the normalized Lebesgue measure on $\Omega$. This representation allows explicit computations with the Siegel-Veech transform restricted to $\mathcal{L}$ rather than to all of $SL(2,\mathbb{R})/\Gamma$.

The following theorem provides a counting formula for the expected number of holonomy vectors of a horizontally short translation surface in subsets of $\mathbb{R}^2$. Let $J = \{\zeta_1 < \zeta_2 < \cdots\}$ denote the set of heights of vectors in $\Lambda$ above the $x$--axis. Fix $s_0 \in (0,1)$ and set
\[
\Omega^0 = \{ (s,t) \in \Omega : s_0 \leq s \leq 1 \}.
\]

\begin{theorem}\label{Prop Points in box}
Let $A = [a,b) \times (0,c)$, where $b - a < \zeta_1 \alpha s_0$. Then
\[
\frac{2(b-a)}{\alpha} I_2
\ \leq\
\int_{\Omega^0} \#(p_{s,t}\Lambda \cap A)\, dm(s,t)
\ =\
\frac{2(b-a)}{\alpha} (I_1 + I_2 + I_3),
\]
where $I_1$, $I_2$, and $I_3$ are explicit functions of $c$ and $s_0$ and $p_{s,t}=\begin{bmatrix}
    s& t\\ 0& s^{-1}
\end{bmatrix}.$
\end{theorem}

Explicit expressions for $I_1$, $I_2$, and $I_3$ are given in \S\ref{Sec: L^1 bounds}. Proposition~\ref{Prop Points in box} stands in contrast with the classical formulas of Siegel \cite{Sie} and Veech \cite{Vee}, where the integral only depends on the measure of $A$ and not its shape or position in $\R^2$. Another feature that contrasts with the usual behavior of the distribution of points in the  $SL(2,\R)$-orbit of $\Lambda$ is that when $\Gamma$ has one cusp, there exists a constant $c_0>0$ dependent on $\Gamma$ and $V$ such that if $c<c_0$, then $p_{s,t}\Lambda\cap A$ is empty. It would be interesting to study the measure on $\R^2$ induced by the Siegel--Veech transform conditioned on $\mathcal{L}$ in more detail.

\subsection{Acknowledgements}
The author thanks Yitwah Cheung, Ryan Bushling, and Bohan Yang for their many comments and suggestions.

\section{Preliminaries}

We will use this section to describe our notation and provide the necessary background.

\subsection{Horocycle and Geodesic Flows}
Let $\Gamma$ be a lattice in $SL(2,\R)$, that is, $\Gamma$ is a discrete subgroup of $SL(2,\R)$ and $SL(2,\R)/\Gamma$ has finite volume.
There are three flows in $SL(2,\R)/\Gamma$ that we are interested in: the geodesic flow and the stable and unstable horocycle flows.
Define the following three 1-parameter subgroups of $SL(2,\R)$.

$$A=\left\{g_t=\begin{bmatrix}
    t& 0\\ 0& t^{-1}
\end{bmatrix}: t>0\right\}, \mbox{  }
N=\left\{h_t=\begin{bmatrix}
    1& t\\0&1
\end{bmatrix}: t\in \R\right\}, \mbox{ and } U=\left\{u_t=\begin{bmatrix}
    1& 0\\-t&1
\end{bmatrix}: t\in \R\right\}.$$

The flows generated on $SL(2,\R)/\Gamma$ by left multiplication by $A$, $N$, and $U$ are called the geodesic, stable, and unstable horocycle flows, respectively. Both $A$ and $N$ preserve the horizontal direction in $\R^2$. In particular, the action by $N$ preserves the collection of horizontally short translation surfaces in the $SL(2,\R)$-orbit of the translation surface $\omega$ since the $x$-axis is an eigenspace with eigenvalue $1$. On the other hand, $A$ does not necessarily preserve the set of horizontally short vectors since its elements stretch the $x$-axis when $t>1$. However, when $t\leq 1$, $g_t$ contracts the $x$-axis by a factor of $t$ and therefore $g_t$ maps the set of horizontally short translation surfaces into itself, but not onto unless $t=1$. It will be convenient for us to consider the elements of $AN$ that send the set of horizontally short lattices into itself. To that end, we define $p_{s,t}=\begin{bmatrix}
    s & t\\ 0 &s^{-1}
\end{bmatrix}\in AN.$

$U$ gives a transversal flow to $\mathcal{L}$. $U$ also has the property that for any $g\Gamma\in\mathcal{L}$, the orbit $\{u_t(g\Gamma):t>0\}$ intersects $\mathcal{L}$ in an infinite discrete set of times \cite{Ath}.

\subsection{Translation surfaces}

A translation surface is a tuple $(X,\omega)$ where $X$ is a compact Riemann surface and $\omega$ is a holomorphic 1-form on $X$. For simplicity, we will abuse this notation and refer to $\omega$ as the translation surface from now on. A way to visualize a translation surface is as a disjoint union of polygons $P=\bigsqcup_{i=1}^n P_i \subset \C$ where the sides of $P$ are partitioned into pairs $\{s_a,s_b\}$ and there is a Euclidean translation mapping $s_a$ onto $s_b$. This procedure provides a closed Riemann surface $P/\sim$ when we glue the sides of $P$ along these identifications. Since $dz$ is invariant under Euclidean translation, it follows that $dz$ descends to a holomorphic one-form on $P/\sim$. This procedure, therefore, generates a translation surface. Every translation surface arises from a construction like this one. For more information on translation surfaces, one may look at the survey by Zorich \cite{Zor}.

There is a natural $SL(2,\R)$-action on the space of translation surfaces. We now describe this action via the polygon definition. Let $\omega$ be a translation surface arising from a finite union of polygons $P\subset\C\simeq \R^2$. Let $g\in SL(2,\R)$, and consider $gP$, the image of $P$ under $g$. $gP$ generates a new translation surface, which we call $g\omega$. Two translation surfaces $(X_1,\omega_1)$ and $(X_2,\omega_2)$ are equivalent if there exists a biholomorophism $f: X_1\rightarrow X_2$ with $f_{*}\omega_1=\omega_2.$ 

\vspace{0.5cm}

\paragraph*{\bf Saddle connections}

A saddle point of $\omega$ is a point $p\in X$ such that $\omega_p=0$. $\omega$ induces a flat metric on $X\setminus\{\text{ saddle points of }\omega\}$. A saddle connection is a geodesic $\gamma$ connecting two (not necessarily distinct) saddle points of $\omega$ without crossing any other saddle points in between. Each saddle connection $\gamma$ provides us with a complex number via $\gamma\mapsto z_\gamma=x_\gamma+iy_\gamma=\int_\gamma \omega.$ We call $z_\gamma$ the holonomy vector of $\gamma$. We collect the set of holonomy vectors of $\omega$ in a set $\Lambda^*=\Lambda^*_\omega$. We will mainly be interested in the shortest element of $\Lambda^*$ in each direction, so we collect these elements in a set $\Lambda=\Lambda_\omega.$
Notice that we have an equivariant relation where for each $g\in SL(2,\R)$, $\Lambda_{g\omega}=g\Lambda_\omega.$

%Masur showed in~\cite{Mas} that for any $\omega$, there exists positive constants $c_1$ and $c_2$ (only dependent on $\omega$) such that
%$$c_1R^2\leq\abs{\Lambda_\omega \cap B(0,R)}\leq c_2R^2,$$
%characterizing the growth rate of the $\Lambda^*_\omega$.
Much work has been done to understand the statistics of the distribution of saddle connections for different types of translation surfaces. Masur showed that the growth rate of the number of saddle connections of length at most $R$ grows quadratically in $R$ \cite{Mas}. More recently, Osman-Southerland-Wang\cite{OsmSouWan} computed effective bounds on the distribution of slope gaps between saddle connections on lattice surfaces, and Artiles \cite{Art} computed the distribution of holonomy vectors of lattice surfaces in randomly chosen directions.

\vspace{0.5cm}

\paragraph*{\bf Lattice Surfaces} %There is a natural $SL(2,\R)$ action on the space of translation surfaces. We described this action via the polygon definition. Let $\omega$ be a translation surface arising from a polygon $P\subset\C\simeq \R^2$. Let $g\in SL(2,\R)$, and consider $gP$, the image of $P$ under $g$. $gP$ generates a new translation surface, which we call $g\omega$. Two translation surfaces $(X_1,\omega_1)$ and $(X_2,\omega_2)$ are equivalent if there exists a biholomorophism $f: X_1\rightarrow X_2$ with $f_{*}\omega_1=\omega_2.$ 

We refer to the stabilizer of $\omega$ under the $SL(2,\R)$-action as the Veech group of $\omega$, and we will denote it by $\Gamma=\Gamma_\omega$. If $\Gamma$ is a lattice in $SL(2,\R)$, we say that $\omega$ is a lattice surface. Veech surface is another common term used for lattice surfaces.%We will restrict our view to the case when $-\Id\in\Gamma$ and $G/\Gamma$ has only one cusp.

Veech showed in \cite{Vee} that if $\omega$ is a lattice surface, we can write $\Lambda$ as a disjoint union of $\Gamma$-orbits. That is, there exist vectors $v_1, v_2,..., v_m$ such that $\Lambda=\bigsqcup_{i=1}^m \Gamma v_i.$

%In the construction of the $\omega$-BCZ map, we will be mainly interested in the shortest vector in any given direction, so let $\Lambda=\Lambda_\omega=\{w\in \Lambda: \abs{w}\leq \abs{w'} \text{ whenever } w \text{ and } w' \text{ are parallel}\}.$ Notice that for all $g\in SL(2,\R)$, $g\Lambda_\omega=\Lambda_{g\omega}.$ Masur and Veech's theorems also apply to $\Lambda$, albeit, there the constants in the theorem might be different.

\subsection{The $\omega$-BCZ map}
Let $\Gamma$ be a lattice in $SL(2,\R)$ containing $-\Id$. Suppose further that $SL(2,\R)/\Gamma$ has only one cusp.

Let $v_1,...,v_m$ be a collection of non-zero vectors in $\R^2$ such that $\Gamma v_i$ is discrete and that the set of directions of $\Gamma v_i$ and $\Gamma v_j$ are disjoint whenever $i\neq j$.

By acting via some element of $SL(2,\R)$ we may assume that $v_1=(1,0)^{\mathsf{T}}$.

Since $\Gamma (1,0)^{\mathsf{T}}$ is discrete it follows that there exists $\alpha>0$ such that 
$\begin{bmatrix}
    1& \alpha\\ 0& 1
\end{bmatrix}$ generate the maximal parabolic subgroup of $\Gamma$ fixing $(1,0)^{\mathsf{T}}$.

%\textcolor{blue}{(We can easily generalized all these result to finitely many cusps)}

In the case that $\Lambda=\bigsqcup_{i=1}^m\Gamma v_i$  is the set of visible holonomy vectors of $\omega$, for $g\in SL(2,\R)$, we have that $g\Lambda=\bigsqcup_{i=1}^m g\Gamma v_i$ is the set of visible holonomy vectors of $g\omega$. Notice that the $SL(2,\R)$-orbit of $\Lambda$ is parameterized by $SL(2,\R)/\Gamma$. We will use the identification $g\Lambda \leftrightarrow g\Gamma$ implicitly during the rest of the paper.

An element $g\Gamma\in SL(2,\R)/\Gamma$ is called $h$-horizontally short if $g\Lambda \cap \{(x,0)\in\R^2:0<x\leq h\}\neq\varnothing$. When $h=1$, we will use the term horizontally short instead of $1$-horizontally short. We denote the set of $h$-horizontally short elements of $SL(2,\R)/\Gamma$ by $\mathcal{L}_h$

Athreya \cite{Ath} showed that $\mathcal{L}_h$ is a Poincar\'e section for the unstable horocycle flow on $SL(2,\R)/\Gamma$. We denote the the return map $T_h:\mathcal{L}_h\rightarrow\mathcal{L}_h$, and $T_1$ will be denoted simply by $T$. The $SL(2,\R)$ invariant probability measure on $SL(2,\R)/\Gamma$ induces a measure $m_h$ on $\mathcal{L}_h$. In \cite{AthChaLel}, Athreya-Chaika-Leli\`evre showed that $\mathcal{L}_h$ is a hypersurface in $SL(2,\R)/\Gamma$ and computed an explicit parameterization. \begin{equation}\label{Eq: Parameterization of BCZ domain}
    \Omega_h=\Omega_{\omega,h}=\{(s,t)\in \R^2: 0<s\leq h, 1-\alpha s<t\leq 1\}\leftrightarrow \mathcal{L}_h
\end{equation}
via $(s,t)\leftrightarrow \begin{bmatrix}
    s & t\\ 0 & s^{-1}
\end{bmatrix}\Gamma$.

We should note here that Uyanik-Work \cite{UyaWor} computes a parameterization of $\mathcal{L}$ for general $\Gamma$ as a disjoint set of triangles in $\R^2$; one triangle per cusp of $SL(2,\R)/\Gamma$.
The identification of $\mathcal{L}_h$ and $\Omega_h$ identifies the measure $m_h$ on $\mathcal{L}_h$ with the normalized Lebesgue measure on $\Omega_h$. For this reason, we will denote both measures as $m_h$. $m_1$ will be simply denoted $m$.

\section{Three Lemmas}\label{Sec: Two Lemmas}

This section contains three key lemmas. The first is a lemma due to Chueng-Quas \cite{CheQua} which provides a way to verify if an ergodic transformation is weakly-mixing. The other two lemmas are computational tools that will be used to count the number of points in subsets of $\R^2$.

\begin{Lemma}(Cheung-Quas 2024)\label{Lem: Cheung-Quas}
    Let $(X,d)$ be a compact metric space equipped with a Borel probability measure $m$. Let $T$ be an ergodic measure-preserving transformation of $(X,m)$.

Suppose that there exists $\tau>0$, a sequence of subsets $(X_k)$ of $X$, a sequence of positive integers $N_k$ and a sequence of maps $\phi_k:X\rightarrow X_k$ satisfying the following conditions:
\begin{enumerate}
    \item $m(X_k)\rightarrow 1$.
    \item The return map $T_k$ of $X_k$ is conjugate to $T$ via $\phi_k$.
    \item For all $\delta>0, m\left(\left\{x\in X: d(\phi_k(x),x)>\delta\right\}\right)\rightarrow 0$.
    \item $m\left(\left\{x\in X_k: R_k^{N_k}(x)=N_k+1\right\}\right)>\tau$, where $R_k^{N_k}(x)$ denotes the $n$th return time to $X_k$.
\end{enumerate}
Then $T$ is weakly mixing.
\end{Lemma}

The ergodicity of $(\mathcal{L},T,m)$ was proved in \cite{UyaWor}. The first three conditions of Lemma \ref{Lem: Cheung-Quas} are not hard to check in our case. A large part of this paper is used to prove part $(4)$ holds for our dynamical system $(\mathcal{L}, T, m)$.

The next lemma will be used to compute the average number of elements of $p_{s,t}\Lambda$ with a prescribed height in a subset $A$ of $\R^2$ when proving Theorem \ref{Prop Points in box}.

\begin{Lemma}\label{Lem: average number of points in an interval from a periodic sequence}
    Let $\alpha>0, j\geq 1$, and $a,b\in \R$ with $b>a$. Suppose also that $b-a<\alpha j s$. Consider the points $0\leq x_1<x_2<...<x_k<\alpha j$, and define $\mathfrak{L}_{s,t}=\{sx_i+t j + \alpha j ns: n\in\Z, 1\leq i\leq k\}$. Then
    $$\int_{1-\alpha s}^1 \#\left(\mathfrak{L}_{s,t} \cap (a,b]\right) dt=\frac{k}{j}(b-a).$$
\end{Lemma}

\begin{proof}
    Let $\chi$ be the characteristic function of $(a,b]$. Then

    \begin{equation*}
        \#\left(\mathfrak{L}_{s,t} \cap [a,b)\right)=\sum_{i=1}^k \sum_{n\in\Z}\chi(sx_i+tj+ \alpha jns)
    \end{equation*}

So we have that 

\begin{equation*}
    \int_{1-\alpha s}^1 \#\left(\mathcal{L}_{s,t} \cap [a,b)\right) dt=\sum_{i=1}^k \int_{1-\alpha s}^1 \sum_{n\in\Z}\chi(sx_i+t j+ \alpha jns)dt
\end{equation*}

So we proceed to compute $\int_{1-\alpha s}^1 S_i(t) dt$, where $S_i(t)= \sum_{n\in\Z}\chi(sx_i+t j+ \alpha jns).$

Notice that $S_i(t)$ is $\alpha s$ periodic, that is 
\begin{equation}
  S_i(t)=S_i(t+\alpha s).  
\end{equation}

We now claim that if $n_1,n_2$ are integers and $\chi(sx_i+t j+ \alpha jn_1s)=\chi(sx_i+tj+ \alpha jn_2s)=1,$ then $n_1=n_2.$

Suppose that $1=\chi(sx_i+t j+ \alpha jn_1s)=\chi(sx_i+t j+ \alpha jn_2s)$. Then $$a<sx_i+t j+ \alpha jn_1s,sx_i+tj+ \alpha jn_2s\leq b.$$
This implies that
\begin{align*}
    b-a&>\abs{(sx_i+t j+ \alpha jn_2s) -(sx_i+t j+ \alpha jn_1s)}\\
    &=\alpha j s\abs{n_2-n_1}.
\end{align*}
If $n_1\neq n_2$, then $b-a>\alpha j s\abs{n_2-n_1}\geq \alpha j s$. This contradicts our hypothesis. Therefore, $n_1=n_2$.

It then follows that
\begin{align*}
    \int_0^{\alpha s} S_i(t) dt&=\sum_{i\in \Z} \int_0^{\alpha s} \chi(sx_i+tj+
    \alpha jns)dt\\
    &=\sum_{i\in \Z} \int_0^{\alpha s} \chi_n(t)dt\\
    &=\sum_{i\in \Z} \operatorname{Len}\left([0,\alpha s)\cap  \left[\frac{a-sx_i}{j}-\alpha n s ,\frac{b-sx_i}{j}-\alpha n s\right)\right)\\
    &=\frac{b-a}{j}.
\end{align*}
%\textbf{\textcolor{blue}{Should I add more detail in the last equality?----Review this proof once more to make sure all the symbols are correct}}
Above, $\chi_n$ is the characteristic function of $\left[\frac{a-sx_i}{j}-\alpha n s ,\frac{b-sx_i}{j}-\alpha n s\right)$.

It then follows that 

\begin{align*}
    \int_{1-\alpha s}^1 \#\left(\mathfrak{L}_{s,t} \cap [a,b)\right) dt&=\sum_{i=1}^k \int_{1-\alpha s}^1 \sum_{n\in\Z}\chi(sx_i+t j+ \alpha jns)\\
    &=\sum_{i=1}^k \int_{1-\alpha s}^1 \sum_{n\in\Z}\chi(sx_i+t j+ \alpha jns)\\
    &=\sum_{i=1}^k \frac{b-a}{j}\\
    &=\frac{k}{j}(b-a).
\end{align*}
\end{proof}
%\begin{Remark}
 %   The hypothesis that $b-a<\alpha js$ is purely for computational purposes in the proof above. If $b-a$ is larger than $\alpha j s$, one can subdivide $[a,b)$ into sufficiently small intervals and compute the integral on each smaller interval and add the result of each integral to obtain the general result for any choice of interval $[a,b]$.
%\end{Remark}

The next lemma is used to control the asymptotics of the functions $I_1$, $I_2$, and $I_3$ from Theorem \ref{Prop Points in box} needed for Corollary \ref{cor: first bound}.

\begin{Lemma}\label{Lem: Abel on Holonomy Vectors}
Let $J\subset [1,\infty)$ and let $g: J\rightarrow \R$ be a non-negative function and $J_N=J\cap [0,N)$. Suppose that \begin{equation}
    \lim_{N\rightarrow \infty} N^{-2}\sum_{j\in J_N} g(j)=c>0.
\end{equation}
then,
\begin{equation}
    \lim_{N\rightarrow \infty} N^{-1}\sum_{j\in J_N} \frac{g(j)}{j}=2c.
\end{equation}
    
\end{Lemma}

\begin{proof}
This will be an application of Abel's summation by parts formula.

Notice that 
\begin{align}
    \sum_{j\in J_N}\frac{g(j)}{j}&=\sum_{n=1}^{N-1}\sum_{n\leq j<n+1} \frac{g(j)}{j}\\
    &\leq \sum_{n=1}^{N-1}\frac{1}{n}\sum_{n\leq j<n+1} g(i)\\
    &=\sum_{n=1}^{N-1}\frac{1}{n(n+1)}\sum_{j\leq n+1} g(j) +\frac{1}{N}\sum_{j<N}g(j)
\end{align}

Notice that $\frac{1}{n(n+1)}\sum_{j<n+1}g(j)$ converges to $c$ as $n$ tends to infinity. So, the Cez\`aro average also converges to $c$. This means that 
\begin{align*}
    \limsup_{N\rightarrow\infty}\frac{1}{N}\sum_{j\in J_N}\frac{g(j)}{j}& \leq \limsup_{N\rightarrow \infty} \left(\frac{1}{N}\sum_{n=1}^{N-1}\frac{1}{n(n+1)}\sum_{j\leq n+1} g(j) +\frac{1}{N^2}\sum_{j<N}g(j)\right)\\
    &=c+c=2c.
\end{align*}

Similarly, since

\begin{align*}
    \sum_{j\in J_N}\frac{g(j)}{j}=\sum_{n=1}^{N-1}\sum_{n\leq j<n+1} \frac{g(j)}{j}
    \geq \sum_{n=1}^{N-1}\frac{1}{n+1}\sum_{n\leq j<n+1} g(j),
\end{align*}

Using the same argument, one can show that
\begin{align}
    2c=\liminf_{N\rightarrow \infty} \sum_{n=1}^{N-1}\frac{1}{n+1}\sum_{n\leq j<n+1} g(j)\leq \liminf_{N\rightarrow\infty} \sum_{j\in J_N}\frac{g(j)}{j}.
\end{align}

This completes the proof.
\end{proof}

\section{Verifying Cheung-Quas Lemma Hypotheses 1-3}\label{sec: hyp 123}

We use this section to verify the hypotheses (1), (2), and (3) from Lemma \ref{Lem: Cheung-Quas}. 

Let $a_n$ be a sequence of positive real numbers strictly increasing to $1$. We consider the sequence of subsets $\mathcal{L}_{a_n}$ equipped with the maps $\phi_n:\mathcal{L}\rightarrow\mathcal{L}_{a_n}$ given by $\phi_n(g\Lambda)=\begin{bmatrix}
    a_n & 0\\ 0& a_n^{-1}
\end{bmatrix}g\Lambda.$

\paragraph*{\bf Part (1)} We must check that the measure of $\mathcal{L}_{a_n}$ converges to $1$ as $n\rightarrow \infty$. This follows from the computation that 

$$m(\mathcal{L}_{a_n})=m(\{(s,t)\in \Omega: 0<s\leq a_n\})=a_n^2.$$

This shows that the sets $\mathcal{L}_{a_n}$ satisfy hypothesis (1).

\vspace{0.5cm}

\paragraph*{\bf Part (2)}
Let $g\Gamma\in \mathcal{L}$ and let $r_0$ denote the return time of $g\Gamma$ to $\mathcal{L}$, that means $T(g\Gamma)=u_{r_0}(g\Gamma)$. Let $v$ be the vector of smallest slope in $g\Lambda\cap((0,1]\times\R_+).$ It then follows that $r_0=\operatorname{slope}(v)$. Since $\operatorname{slope}(\phi_n(v))=\frac{1}{a_n^2}\operatorname{slope}(v)$, it follows that the return time of $\phi_n(g\Gamma)$ to $\mathcal{L}_{a_n}$ is $a_n^{-2}r_0$. This means $T_{a_n}(\phi_n(g\Gamma))=u_{a_n^{-2}r_0}\phi_n(g\Gamma)$. Since

$$\begin{bmatrix}
    a_n&0\\0&a_n^{-1}
\end{bmatrix}\begin{bmatrix}
    1&0\\-r_0&1
\end{bmatrix}=\begin{bmatrix}
    1& 0\\-a_n^{-2}r_0& 1
\end{bmatrix}\begin{bmatrix}
    a_n&0\\0&a_n^{-1}
\end{bmatrix},$$ it follows that $\phi_n\circ T=T_{a_n}\circ\phi_n$.

\vspace{1cm}

%We must show the following identity: $T_{s_k}\circ\phi_k=\phi_{k}\circ T$.

%Notice that the amount of time for $g\Gamma$ to return to $\mathcal{L}_h$ is the smallest slope of a vector in $g\Lambda\cap\{(x,y)\in \R^2: 0<x\leq h, 0<y \}$. If $v\in\R^2\setminus\{0\}$, then $$\operatorname{slope}\left(\begin{bmatrix}
%    s_k & 0\\ 0& s_k^{-1}
%\end{bmatrix}v\right)=\frac{1}{s_k^2}\operatorname{slope}(v).$$

%Let $g\Gamma\in\mathcal{L}$. Then there exist $(s,t)\in \Omega$ such that $g\Gamma=\begin{bmatrix}
 %   s & t\\ 0 & s^{-1}
%\end{bmatrix}\Gamma.$ Then $\phi_k(g\Gamma)=\begin{bmatrix}
%    ss_k & ts_k\\ 0 & s^{-1}s_k^{-1}
%\end{bmatrix}\Gamma.$

%Let $(s,t)\in \Omega$ and $h>0$. Then,

%\begin{equation}\label{Eq Conjugation of Return times}
%    \begin{bmatrix}
%        h & 0\\ 0& h^{-1}
    %\end{bmatrix}^{-1}\begin{bmatrix}
     %   1 & 0\\ -\frac{1}{sth^2} & 0
    %\end{bmatrix}\begin{bmatrix}
     %   h & 0\\ 0& h^{-1}
    %\end{bmatrix}=\begin{bmatrix}
     %   1 & 0\\ -\frac{1}{st} & 1
    %\end{bmatrix}
%\end{equation}

%Since the return time to $\mathcal{L}_{s_k}$ of $\phi_k\left(\begin{bmatrix}
%    s & t\\ 0 & s^{-1}
%\end{bmatrix}\right)=\frac{1}{st s_k^2}$, this completes the proof of conjugacy.

\paragraph*{\bf Part (3)}
Consider the triangle $\Delta_n$ with vertices $(1,a_n)^{\mathsf{T}} , (1-a_n^2,a_n)^{\mathsf{T}} , (1, a_n-\alpha a_n^2)^{\mathsf{T}}$ in $\Omega$. A computation shows that if $(x,y)^{\mathsf{T}}$ is in the interior of $\Delta_n$, then $d((x,y)^{\mathsf{T}} ,\phi_n((x,y)^{\mathsf{T}} ))^2=(1-a_n)^2(x^2+\frac{y^2}{a_n^2})$. As $n\rightarrow \infty$, the distance converges to zero. Since $m(\Delta_n)\rightarrow 1$, we have completed the proof of part 3.

\vspace{1cm}
All that remains to show is that part (4) of Lemma \ref{Lem: Cheung-Quas} holds for our dynamical system $(\mathcal{L}, T, m)$.

\section{Heights of Holonomy Vectors} \label{Sec: Heights of Holonomy Vectors}

%\subsection{Veech groups with one cusp}\label{Subsec: Veech groups with one cusp}

%We will begin by working out the case where $G/\Gamma$ has only one cusp. \textcolor{red}{The general case will be worked out on \ref{Subsec: Veech groups with finitely many cusps}.} Since $\omega$ has a horizontally short lattice, it follows that there exists an $\alpha>0$ such that $h_\alpha=\begin{bmatrix}
%    1 & \alpha\\ 0 & 1
%\end{bmatrix}$ generates the maximal parabolic subgroup of $\Gamma$ fixing the horizontal direction. We then define the set $J=J_\omega=\pi_2(\Lambda(\omega))\cap \R_+$, i.e. the set of positive heights of holonomy vectors of $\omega$, and we denote by $J_N=J\cap (0,N)$.

In this section we study the set of heights vectors in $\Lambda$. Let $J=\pi_2(\Lambda)\cap \R_+$, where $\pi_2$ is the projection into the $y$-axis. 
It will be convenient for our computations to think of $J$ as an ordered set, $J=\{\zeta_1<\zeta_2<...\}$. For each $R>0$, we define $J_R=J\cap(0,R)$. The next two propositions tell us that $J$ is discrete and that it has bounded gaps.

%Let $\omega$ be a horizontally short Veech surface. Define $J=J_\omega=\pi_2(\Lambda_\omega)\cap \R_+$, where $\pi_2$ is the projection into the y-axis. It will be convenient for our computations to think of $J$ as an ordered set, $J=\{\zeta_1<\zeta_2<...\}$. Also, for each $R>0$, we define $J_R=J\cap(0,R)$. The next two Propositions tell us that $J$ is discrete and that it has bounded gaps.

\begin{Proposition}\label{Prop Discreteness of heights}
 The set $J$ is a discrete subset of the positive $y$-axis.
\end{Proposition}

\begin{proof}

Since $\begin{bmatrix}
    1 & \alpha\\ 0 & 1
\end{bmatrix}\Lambda=\Lambda$, it follows that the set of heights in $J_R$ are precisely the set of heights of the vectors in the triangle $T$ bounded by $0<x\leq \alpha y$ and $0<y<R$. Since $\Lambda$ is discrete and $T$ is bounded, it follows that $\Lambda\cap T$ is finite. Hence $J_R$ is finite and, therefore, $J$ is discrete.

    %Since $\omega$ has a horizontal saddle connection we have a cylinder decomposition of $\omega$ into horizontal cylinders with heights $h_1,h_2,...,h_k$ for some finite positive integer $k$.[MASUR] It then follows that $J\subset \N[h_1,...,h_k]$. Since $\N[h_1,...,h_k]$ is a discrete subset of $\R_+$, it follows that $J$ is discrete.
\end{proof}

\begin{Proposition}\label{Prop: bound on height gaps}
Let us order $J=\{\zeta_1<\zeta_2<...\}$ with the usual ordering in $\R$. Then $\sup_{n\in\N} \{\zeta_{n+1}-\zeta_n\}<\infty$.

\end{Proposition}

\begin{proof}
Let $(x,0)^{\mathsf{T}} \in\Lambda$. Let $S$ be the stabilizer of $(x,0)^{\mathsf{T}} $ in $\Gamma$. Choose $A=\begin{bmatrix}
    a & b\\ c & d
\end{bmatrix}\in\Gamma\setminus S$ with $c\neq 0$. We know $S=\left\langle\begin{bmatrix}
    1& \alpha\\ 0& 1
\end{bmatrix}\right\rangle.$

Then consider $ASA^{-1}(x,0)^{\mathsf{T}} =\{x(1-n\alpha ac,-n\alpha c^2)\in \R^2: n\in \Z\}$. This means that $$J\supset \pi_2\left(ASA^{-1}(x,0)^{\mathsf{T}} \right)\cap \R_+=\{xn\alpha c^2:n\in \N\}.$$ This shows that $\sup_{n\in\N} \{\zeta_{n+1}-\zeta_n\}\leq \alpha xc^2.$

\end{proof}

Next, we define a function on the set $J$ which will help count the number of points in $\Lambda$ with the same height up the action of $\begin{bmatrix}
    1& \alpha\\ 0& 1
\end{bmatrix}$ on $\Lambda$.

\begin{Definition}\label{Def: geometric Euler Totient function}
    Given $\Lambda$ as above, we define $\phi:J\rightarrow\R$ by 
    \begin{equation}
        \phi(j)=\#\left\{x\in \R: (x,j)\in \Lambda \text{ and } 0\leq x< \alpha j\right\}
    \end{equation}
    
    \end{Definition}

In particular notice that if $\Gamma=SL(2,\Z)$ and $\Lambda=SL(2,\Z)(1,0)^{\mathsf{T}}$, then $\alpha=1$, $J=\N$ and $$\phi(j)=\#\left\{x\in \R: (x,j)\in \Lambda \text{ and } 0\leq x< \alpha j\right\}=\#\left\{n\in \N: \gcd(n,j)=1 \text{ and } 0\leq n<j\right\}$$
Which is exactly the classical Euler-totient function.
%Throughout this section since $\omega$ is fixed, we will denote $\phi_\omega$ as $\phi$ to reduce the complexity of the notation.

Notice also that since $\begin{bmatrix}
    1 & \alpha\\ 0 & 1
\end{bmatrix}\in \Gamma$, any horizontal segment $[a,b)\times \{j\}$ where $b-a=\alpha j$, contains exactly $\phi(j)$ points.

The following lemma will give a different way to compute $\phi$ in terms of determinants of pair of vector in $\Lambda\times\Lambda.$ We will use this interpretation of $\phi$ in \S\ref{L^2 bounds}.

\begin{Lemma}
    Consider the following equivalence relation on $\Lambda\times\Lambda$, $(u_1,u_2)\sim(v_1,v_2)$ if there exists $\gamma\in \Gamma$ such that $\gamma u_1=v_1$ and $\gamma u_2=v_2$. Define
    $D_j=\{(u_1,u_2)\in \Lambda\times\Lambda/\sim:\abs{\operatorname{det}(u_1,u_2)}=j \text{ and } (u_1,u_2)\sim(e_1,v) \text{ for some } v\in\Lambda\}.$ Then, $\#D_j=2\phi(j)$ if $j\in J$, $\#D_0=2$, and $D_j=0$ otherwise.
\end{Lemma}

\begin{proof}
    Let $j\geq0$ and suppose that $(u_1,u_2)\in D_j$. Choose $\gamma\in \Gamma$ such that $\gamma u_1=e_1$. Since $j=\abs{\det(u_1,u_2)}=\abs{\det(e_1,\gamma u_2)}= \abs{\pi_2(\gamma u_2)}$, it follows that if $j\notin J$, then $D_j=0$. If $j=0$, it follows that $e_1$ and $\gamma u_2$ are collinear. This then implies $\gamma u_2=\pm e_1.$ Therefore, $D_0=2$. Finally, if $j\in J$, then $\gamma u_2$ belongs the horizontal lines $y=\pm j$. By applying $\begin{bmatrix}
        1& \alpha\\ 0 & 1
    \end{bmatrix}$ to $e_1$ and $\gamma u_2$ some integer multiple of times, we get a unique element $(e_1, v)$ where the $x$-coordinate of $v$ is between $0$ and $\alpha j$. We can do a similar procedure for $e_1$ and $-\gamma u_2$. This shows $D_j=2\phi(j)$.
\end{proof}

Next, we study the statistics of the function $\phi$.

\begin{Lemma}\label{cor: Asymtotics of Geometric Euler totitent function}
    Let $\Lambda$ be as above and $R>0$. There is a constant $c_\omega>0$ such that

\begin{enumerate}
    \item $\lim_{R\rightarrow \infty}R^{-2}\sum_{J_{R}}\phi(j)=c_\omega;$
    \item $\lim_{R\rightarrow \infty}R^{-1}\sum_{J_{R}}\frac{\phi(j)}{j}=2c_\omega.$
\end{enumerate}
    
\end{Lemma}

\begin{proof}
Define $T_R=\{(x,y)\in\R^2: 0<y<R, 0\leq x<\alpha y\}$. Notice that $\sum_{j\in J_R}\phi(j)=\#\left(T_R\cap \Lambda\right).$ Veech \cite{Vee} tells us that 
\begin{align*}
\lim_{R\rightarrow\infty}\frac{\sum_{j\in J_R}\phi(j)}{R^2}&=\lim_{R\rightarrow\infty}\frac{\#\left(T_R\cap \Lambda\right)}{R^2}= c_\omega>0,
\end{align*}

where $c_\omega$ is a constant only dependent on $\omega$.

The second statement follows from the first part together with Lemma \ref{Lem: Abel on Holonomy Vectors}.
    
\end{proof}

\section{Reduction of Fourth Condition}

%The fourth condition from Lemma \ref{Lem: Cheung-Quas} tells us to look at the the following: take the $u_r$-orbit of a point $x\in \mathcal{L}_{a_k}$ until it has intersected $\mathcal{L}_{a_k}$ exactly $N_k$ times. We are interested in those points $x\in \mathcal{L}_{a_k}$ where we hit $\mathcal{L}\setminus \mathcal{L}_{a_k}$ exactly once before arriving at $\mathcal{L}_{a_k}$ for the $N_k$ time.

According to the fourth condition of Lemma~\ref{Lem: Cheung-Quas}, we need to bound the measure of the set of those $g\Gamma\in \mathcal{L}_{a_n}$ with the property that before returning to $\mathcal{L}_{a_n}$ exactly $N$ times, they visited $\mathcal{L}\setminus\mathcal{L}_{a_n}$ exactly once.
%we consider the $U$-orbit of a point $g\Gamma \in \mathcal{L}_{a_n}$ until it has intersected 
%$\mathcal{L}_{a_n}$ precisely $N_n$ times. 
%Our focus is on those $g\Gamma \in \mathcal{L}_{a_n}$ whose orbit meets 
%$\mathcal{L} \setminus \mathcal{L}_{a_n}$ exactly once before its $N_n$-th return to $\mathcal{L}_{a_n}$.

In the context of the $\omega$-BCZ-map we have a geometric viewpoint provided by Cheung-Quas \cite{CheQua}.

\begin{Lemma}\label{Lem: fourth condition identifies}
    Let $g\Gamma\in \mathcal{L}_h$. Let $s_1<s_2<...$ be the slopes of vectors in $g\Lambda \cap \{(x,y)\in \R^2: 0<x\leq h, 0<y\}$. The $N$th return time of $g\Gamma$ to $\mathcal{L}_h$ is the $(N+1)$th return time of $g\Gamma$ to $\mathcal{L}$ if and only if $\#\{(x,y)\in g\Lambda: h<x\leq1, y<s_Nx\}=1.$
\end{Lemma}

Lemma \ref{Lem: fourth condition identifies} tells us that we need to bound below the measure of lattices in $\mathcal{L}_h$ where we have exactly one point in the rectangle with width $1-h$ and height approximately equal $s_N$. It follows from the ergodicity of the unstable horocycle flow that $\frac{1}{N}s_N$ converges to $2c_\omega$ when $N\rightarrow\infty.$ With this in mind, we define the following quantities:

Let $\beta>0$ and define $c_n=\frac{\beta}{1-a_n}$, where $a_n$ are as in \S\ref{sec: hyp 123}. Define $N_n=\floor{c_n}$. Then by the argument above, for $n$ large, the slope of the $N_n$-th vector of $g\Lambda$ is approximately $2c_\omega N_n.$ Therefore, the set $D_{0,n}=\{(s,t)\in\Omega_{a_n}: \text{The } N_n\text{-th horocycle return time of } g_{s,t}\Lambda \text{ to } \mathcal{L}_{a_n} \text{ is between } c_\omega N_n \text{ and } 4c_\omega N_n\}$ has measure converging to $1$ as $n$ tends to $\infty$.

Define $B_{1,n}=(a_n,1]\times (0, c_\omega c_n)$ and $B_{2,n}=(a_n,1]\times (0, 4c_\omega c_n)$.

In \S\ref{Sec: L^1 bounds}, we will show that the average number of points $B_{1,n}$ is bounded below by a constant multiplied by the area of $B_{1,n}$ and, in \S\ref{L^2 bounds}, we will show that the event that $B_{2,n}$ contains two or more points is extremely rare when we restrict our view to compact subsets of $\Omega$. To be more precise, we show the following two claims:

Let $F_{1,n},F_{2,n}:\Omega\rightarrow\R$ be given by $F_{1,n}(s,t)=\#(p_{s,t}\Lambda \cap B_{1,n})$ and $F_{2,n}(s,t)=\#(p_{s,t}\Lambda \cap B_{2,n})$.

\begin{Claim}\label{Claim: small box}
    Let $s_0\in (0,1)$. Define $\Omega^{0}=\{(s,t)\in\Omega: s\geq s_0\}.$ Then, there exists $c_1>0$ such that for $n$ large enough,
    $$\int_{\Omega^0}F_{1,n}(s,t) dm(s,t)\geq c_1\beta.$$
    
\end{Claim}

and

\begin{Claim}\label{Claim: Big box}
    With notation as in Claim \ref{Claim: small box}, for $n$ large enough we have that
    $$\int_{\Omega^0} F_{2,n}(s,t)\mathbb{1}_{\{F_{2,n}(s,t)>1\}}dm(s,t)\leq c_2\beta^2.$$
\end{Claim}

We now use the rest of the section to show how Claims \ref{Claim: small box} and \ref{Claim: Big box} complete the proof of condition 4. 

\begin{Proposition}
    $(\mathcal{L},T,m)$ satisfy Condition (4) from Lemma \ref{Lem: Cheung-Quas}.
\end{Proposition}

\begin{proof}
    This is the same proof provided by Cheung-Quas \cite{CheQua} in the case of the classical BCZ map, but we include it here for completion. Let $$D_{0,n}=\{(s,t)\in\Omega_{a_n}: \text{The } N_n\text{-th horocycle return time of } g_{s,t}\Lambda \text{ to } \mathcal{L}_{a_n} \text{ is between } c_\omega N_n \text{ and } 4c_\omega N_n\}.$$ Since the horocycle flow is ergodic in $SL(2,\R)/\Gamma$ and the average return time to $\mathcal{L}$ is $2c_\omega$, then $m(D_{0,n})$ converges to $1$.
    Define $D_{1,n}=\{(s,t)\in \Omega^0: F_{1,n}(s,t)=1\}$ and $D_{2,n}=\{(s,t)\in \Omega_n: F_{2,n}(s,t)=1\}$.

    Notice that if $(s,t)\in D_{0,n}\cap D_{1,n}\cap D_{2,n}$ then the $N_n-$th return time of $g_{s,t}\Gamma$ to $\mathcal{L}_{a_n}$ is the $(N+1)$-th return time of $g_{s,t}\Gamma$ to $\mathcal{L}$. Therefore, by Lemma \ref{Lem: fourth condition identifies}, it suffices to find a uniform positive lower bound for $m(D_{0,n}\cap D_{1,n}\cap D_{2,n})$ for large enough values of $n$.

Using Claim \ref{Claim: small box} followed by Claim \ref{Claim: Big box}, we get the following lower bound on $m(D_{1,n})$.
    \begin{align*}
       m(D_{1,n})&=\int_{\Omega^0}F_{1,n}\mathbb{1}_{\{F_{1,n}=1\}} dm\\
       &=\int_{\Omega^0}F_{1,n}dm - \int_{\Omega^0} F_1\mathbb{1}_{\{F_{1,n}>1\}}dm\\
       &\geq c_1\beta -\int_{\Omega^0} F_1\mathbb{1}_{\{F_{1,n}>1\}}dm\\
       &\geq c_1\beta -\int_{\Omega^0} F_1\mathbb{1}_{\{F_{2,n}>1\}}dm\\
       &\geq c_1\beta-c_2\beta^2
    \end{align*}

Notice also that $D_{1,n}\cap D_{2,n}= D_1\setminus\{(s,t)\in \Omega^0: F_{2,n}(s,t)>1\}$. Therefore, $m(D_{1,n}\cap D_{2,n})\geq m(D_{1,n})-\int_{\Omega^0}F_{2,n}\mathbb{1}_{F_{2,n}>1}dm\geq c_1\beta-2c_2\beta^2.$ Since $m(D_{0,n})$ converges to $1$, it follows that for $n$ large enough, $m(D_{0,n}\cap D_{1,n}\cap D_{2,n})\geq  c_1\beta-3c_2\beta^2$. By a small enough positive $\beta$, we can ensure that the quantity $c_1\beta-3c_2\beta^2>0$. This completes the proof.
\end{proof}

All that remains now is to show that Claims \ref{Claim: small box} and \ref{Claim: Big box} are true.

\section{Average number of points in thin rectangles}\label{Sec: L^1 bounds}

Fix $s_0\in (0,1)$ and define $\Omega^0=\{(s,t)\in \Omega: s\geq s_0\}.$ Let $a$ and $b$ be real numbers with $b>a$ and let $c>0$. Let $A=[a,b)\times(0,c).$

Our first goal will be to understand how many points are in $A\cap p_{s,t}\Lambda$ where $(s,t)\in \Omega^0$. This is a hard question to answer for any one particular $(s,t)$. Instead we will provide the bounds for the mean on the number of points over all of $\Omega^0$. Let $W(s,t,A)=\#(p_{s,t}\Lambda \cap A)$.

\begingroup
\def\thetheorem{\ref{Prop Points in box}}
\begin{theorem}
    Let $s_0\in (0,1)$ and $A=[a,b)\times (0,c)$, where $b-a<\zeta_1\alpha s_0$. Define $k'=\max\left\{k\in \N:\zeta_k\leq cs_0\right\}$ and $k''=\max\left\{k\in \N: \zeta_k\leq c\right\}$.
    Then
    \begin{equation*}
        \frac{2(b-a)}{\alpha}I_2\leq \int_{\Omega^0}W(s,t,A) dm(s,t)= \frac{2(b-a)}{\alpha}\left( I_1+I_2+I_3\right),
    \end{equation*}

where
\begin{enumerate}
    \item $I_1=\frac{\zeta_{k'+1}-cs_0}{c}\sum_{j\in J_{cs_0}}\frac{\phi(j)}{j},$
    \item %$I_2=\left(\frac{\zeta_{k''}-\zeta_{k'+1}}{c}\right)\frac{\phi(\zeta_1)}{\zeta_1}+\left(\frac{\zeta_{k''}-\zeta_{k'+1}}{c}\right)\sum_{j\in J_{cs_0}}\frac{\phi(j)}{j}+\sum_{j\in J_c\setminus J_{cs_0}}\left(\frac{\zeta_{k''}-j}{c}\right)\frac{\phi(j)}{j},$
    $I_2=\sum_{\ell=1}^{k''-1}\left(\frac{\zeta_{\ell+1}-\zeta_{\ell}}{c}\sum_{j\in J_{\zeta_{\ell+1}}}\frac{\phi(j)}{j}\right)-\sum_{\ell=1}^{k'}\left(\frac{\zeta_{\ell+1}-\zeta_{\ell}}{c}\sum_{j\in J_{\zeta_{\ell+1}}}\frac{\phi(j)}{j}\right),$
    \item $I_3=\left(\frac{c-\zeta_{k''}}{c}\right)\sum_{j\in J_c}\frac{\phi(j)}{j}.$
\end{enumerate}
\end{theorem}
\addtocounter{theorem}{-1}
\endgroup

\begin{proof}
    Define $G(s,t,A,j)=\#\left(p_{s,t}\Lambda\cap A \cap \{y=\frac{j}{s}\}\right).$ It then follows that
\begin{align*}
     \int_{\Omega_{0}}W(s,t,A) dm&= \frac{2}{\alpha} \int_{s_0}^1\int_{1-\alpha s}^1 \#\left(p_{s,t}\Lambda\cap A\right) dt ds\\
     &= \frac{2}{\alpha} \int_{s_0}^1\int_{1-\alpha s}^1 \sum_{j\in J_{cs}}G(s,t,A,j) dt ds\\
     &= \frac{2}{\alpha} \int_{s_0}^1\sum_{j\in J_{cs}}\int_{1-\alpha s}^1 G(s,t,A,j) dt ds
\end{align*}

Using Lemma \ref{Lem: average number of points in an interval from a periodic sequence}, we get that 
$\int_{1-\alpha s}^1 G(s,t,A,j)=\frac{\phi(j)}{j}(b-a).$

This reduces the computation to

$$\int_{\Omega_{0}}W(s,t,A) dm= \frac{2(b-a)}{\alpha} \int_{s_0}^1 \sum_{j\in J_{cs}} \frac{\phi(j)}{j}ds.$$

The rest of this proof will be used to reduce the integral on the right to the sum of the functions $I_1$, $I_2$, and $I_3$.
\begin{equation}\label{Eq: Control of heigh contribution}
    \int_{s_0}^1 \sum_{j\in J_{cs}} \frac{\phi(j)}{j}ds= \int_{s_0}^{\frac{\zeta_{k'+1}}{c}} \sum_{j\in J_{cs}} \frac{\phi(j)}{j}ds +\sum_{\ell=1}^{k''-k-1}\int_{\frac{\zeta_{k'+\ell}}{c}}^{\frac{\zeta_{k'+\ell+1}}{c}} \sum_{j\in J_{cs}} \frac{\phi(j)}{j}ds+\int_{\frac{\zeta_{k''}}{c}}^1 \sum_{j\in J_{cs}} \frac{\phi(j)}{j}ds
\end{equation}

First notice that if $s\in(s_0,\frac{\zeta_{k'+1}}{c})$, then $$\sum_{j\in J_{cs}} \frac{\phi(j)}{j}=\sum_{j\in J_{cs_0}} \frac{\phi(j)}{j}.$$

Therefore,

$$\int_{s_0}^{\frac{\zeta_{k'+1}}{c}} \sum_{j\in J_{cs}} \frac{\phi(j)}{j}ds=\left(\frac{\zeta_{k'+1}-cs_0}{c}\right)\sum_{j\in J_{cs_0}} \frac{\phi(j)}{j}ds= I_1.$$

If, $s\in(\frac{\zeta_{k''}}{c},1)$, then 
$$\sum_{j\in J_{cs}} \frac{\phi(j)}{j}=\sum_{j\in J_c}\frac{\phi(j)}{j}$$

Therefore,

$$\int_{\frac{\zeta_{k''}}{c}}^1 \sum_{j\in J_{cs}} \frac{\phi(j)}{j}ds=\left(\frac{c-\zeta_{k''}}{c}\right)\sum_{j\in J_{c}} \frac{\phi(j)}{j}ds=I_3.$$

Lastly, if $s\in(\frac{\zeta_{k'+\ell}}{c},\frac{\zeta_{k'+\ell+1}}{c}), $ then

$$\sum_{j\in J_{cs}} \frac{\phi(j)}{j}=\sum_{j\in J_{\zeta_{k'+\ell+1}}} \frac{\phi(j)}{j}.$$

This then helps us reindex the remaining terms of Equation \ref{Eq: Control of heigh contribution}.
\begin{align*}
\sum_{\ell=1}^{k''-k-1}\int_{\frac{\zeta_{k'+\ell}}{c}}^{\frac{\zeta_{k'+\ell+1}}{c}}\sum_{j\in J_{\zeta_{k'+\ell+1}}}\frac{\phi(j)}{j}&=\sum_{\ell=1}^{k''-k'-1}\left(\frac{\zeta_{k'+\ell+1}-\zeta_{k+\ell}}{c}\sum_{j\in J_{\zeta_{k'+\ell+1}}}\frac{\phi(j)}{j}\right)\\
&=\sum_{\ell=1}^{k''-1}\left(\frac{\zeta_{\ell+1}-\zeta_{\ell}}{c}\sum_{j\in J_{\zeta_{\ell+1}}}\frac{\phi(j)}{j}\right)-\sum_{\ell=1}^{k'}\left(\frac{\zeta_{\ell+1}-\zeta_{\ell}}{c}\sum_{j\in J_{\zeta_{\ell+1}}}\frac{\phi(j)}{j}\right)\\&=I_2.
\end{align*}
   
\end{proof}

\begin{corollary}\label{Cor: bounding L^1}
    Let $a_n$ be a sequence of positive numbers increasing to $1$. Let $\beta>0$ and define $c_n=\frac{\beta}{1-a_n}$. Define $A_n=[a_n,1)\times (0,c_n)$. Then, with the notation as in Proposition \ref{Prop Points in box}, as $n\rightarrow \infty$,
    \begin{enumerate}
        \item $I_1$ is bounded,
        \item $I_2\sim c_\omega c_n (1-s_0)$
        \item $I_3$ is bounded.
    \end{enumerate}
\end{corollary}

\begin{proof}
    We first proof $(1)$. 
    \begin{align}
        I_1\leq\frac{\zeta_{k'+1}-\zeta_{k'}}{c_n}\sum_{j\in J_{c_ns_0}}\frac{\phi(j)}{j}.
    \end{align}
By Proposition \ref{Prop: bound on height gaps}, we know $\limsup_{c_n\rightarrow\infty}\{\zeta_{k'+1}-\zeta_{k'}\}$ is finite. By Lemma \ref{cor: Asymtotics of Geometric Euler totitent function} we know that $\lim_{n\rightarrow\infty}\frac{1}{c_n}\sum_{j\in J_{c_ns_0}}\frac{\phi(j)}{j}=2c_\omega s_0.$

Therefore, $$\limsup_{c_n\rightarrow\infty} I_1\leq 2c_\omega s_0\limsup_{c_n\rightarrow\infty}\{\zeta_{k'+1}-\zeta_{k'}\},$$ which is a finite number.

The proof that shows that $I_3$ is bounded is identical to the proof that $I_1$ is bounded, so we omit it.

Lastly, we show that $I_2\sim c_\omega c_n (1-s_0)$. Using Proposition \ref{Prop Points in box}, we get the following

\begin{align*}
    \frac{I_2}{c_n}&=\frac{1}{c_n}\sum_{\ell=1}^{k''-1}\left(\frac{\zeta_{\ell+1}-\zeta_{\ell}}{c_n}\sum_{j\in J_{\zeta_{\ell+1}}}\frac{\phi(j)}{j}\right)-\frac{1}{c_n}\sum_{\ell=1}^{k'}\left(\frac{\zeta_{\ell+1}-\zeta_{\ell}}{c_n}\sum_{j\in J_{\zeta_{\ell+1}}}\frac{\phi(j)}{j}\right)\\
    &=\frac{1}{c_n^2}\sum_{\ell=1}^{k''-1}\left(\zeta_{k''}-\zeta_{\ell}\right)\frac{\phi(\zeta_{\ell})}{\zeta_{\ell}}-\frac{1}{c_n^2}\sum_{\ell=1}^{k'}\left(\zeta_{k'}-\zeta_{\ell}\right)\frac{\phi(\zeta_{\ell})}{\zeta_{\ell}}.
    \end{align*}

Proposition $\ref{Prop: bound on height gaps}$ tells us that $\abs{c_n-\zeta_{k''}}$ and ${c_ns_0-\zeta_{k'}}$ are bounded quantities and therefore $\lim\frac{\zeta_{k''}}{c_n}=1$ and $\lim \frac{\zeta_{k'}}{c_ns_0}=1.$ Using Lemma \ref{cor: Asymtotics of Geometric Euler totitent function}, we get the following:

\begin{align*}
\lim_{n\rightarrow\infty} \frac{1}{c_n^2}\sum_{\ell=1}^{k''-1}\left(\zeta_{k''}-\zeta_{\ell}\right)\frac{\phi(\zeta_{\ell})}{\zeta_{\ell}}&= \lim_{n\rightarrow\infty}\left( \frac{\zeta_{k''}}{c_n}\frac{1}{c_n}\sum_{j\in J_{c_n}}\frac{\phi(j)}{j}-\frac{1}{c_n^2}\sum_{j\in J_{c_n}}\phi(j)\right)\\
&=2c_\omega -c_\omega\\
&= c_\omega
\end{align*}

Similarly,

\begin{align*}
    \lim_{n\rightarrow\infty} \frac{1}{c_n^2}\sum_{\ell=1}^{k'}\left(\zeta_{k'}-\zeta_{\ell}\right)\frac{\phi(\zeta_{\ell})}{\zeta_{\ell}}=c_\omega s_0.
\end{align*}

Therefore,
$\lim_{n\rightarrow\infty}\frac{I_2}{c_n}=c\omega(1-s_0).$

This completes the proof.

\end{proof}

\begin{corollary}\label{cor: first bound}
    Let $a_n$, $c_n$, and $A_n$ be as in Corollary \ref{Cor: bounding L^1}. Then,

    $$\lim_{n\rightarrow\infty}\int_{\Omega^0} W(s,t,A_n) dm= \frac{2}{\alpha}c_\omega\beta(1-s_0).$$
\end{corollary}

Claim \ref{Claim: small box} follows from Corollary \ref{cor: first bound}.

\section{Bound for Second Moment}\label{L^2 bounds}

Our next goal is to bound $\int_{\Omega^0} W(s,t,A_n)\mathbb{1}_{W(s,t,A_n)>1}dm$, where $A_n$ is a rectangle as in Corollary \ref{Cor: bounding L^1}. Let $A$ be a rectangle as in Proposition \ref{Prop Points in box}. Since $W(s,t,A)$ is a non-negative integer, it follows that $W(s,t,A)\mathbb{1}_{W(s,t,A)>1}\leq W^2(s,t,A)-W(s,t,A)$. So we will focus on controlling $\int_{\Omega^0}W^2(s,t,A_n)-W(s,t,A_n)dm.$ Define $H(r,s,t,A)=\#(u_rp_{s,t}\Lambda\cap A)$. Athreya-Chaika \cite{AthCha2} showed that the return times of elements of $\mathcal{L}$ with respect to $u_r$ are bounded below by some positive constant $r_*$. Define the following two sets:

$$K_+=\left\{u_rp_{s,t}\Gamma\in SL(2,\R)/\Gamma: r\in\left[0,\frac{r_*}{2}\right], (s,t)\in \Omega^0\right\}$$

and 

$$K_-=\left\{u_rp_{s,t}\Gamma\in SL(2,\R)/\Gamma: r\in\left[-\frac{r_*}{2},0\right], (s,t)\in \Omega^0\right\}.$$

Notice that for $(s,t)\in\Omega^0\cap \Omega_{a_n}$ and $r\in[0,\frac{r_*}{2}]$, $W(s,t,A_n)\leq H(r,s,t,A_n)$ and that for $(s,t)\in\Omega^0\setminus\Omega_{a_n}$ and $r\in[-\frac{r_*}{2},0]$, then $W(s,t,A_n)\leq H(r,s,t,A_n)$ . Since both $W(s,t,A_n)$ and $H(r,s,t,A_n)$ are non-negative integer-valued functions, the following lemma follows:

\begin{Lemma}\label{lem up and down horocycle}
    Let $A_n$, $K_+$, and $K_-$ be defined as above. Then
    \begin{enumerate}
        \item For $(s,t)\in\Omega^0\cap \Omega_{a_n}$ and $r\in[0,\frac{r_*}{2}]$, 
        $$W^2(s,t,A_n)- W^2(s,t,A_n)\leq H^2(r,s,t,A_n)-H(r,s,t,A_n);$$
        \item For $(s,t)\in\Omega^0\setminus\Omega_{a_n}$ and $r\in[-\frac{r_*}{2},0]$,
        $$W^2(s,t,A_n)- W^2(s,t,A_n)\leq H^2(r,s,t,A_n)-H(r,s,t,A_n).$$
    \end{enumerate}
\end{Lemma}

Define $K=K_+\cup K_-$. Then We get the following bound:

\begin{Lemma}\label{Lem: thikening domain}
    For $A_n$ defined as in Corollary\ref{cor: first bound},
    Then, $$\int_{\Omega^0}W^2(s,t,A_n)-W(s,t,A_n) dm\leq \frac{2}{r_{*}}\int_{K}H^2(r,s,t,A_n)-H(r,s,t,A_n)drdm.$$
\end{Lemma}

\begin{proof}
    Lemma \ref{lem up and down horocycle} $(1)$ implies that if $(s,t)\in\Omega^0\cap \Omega_{a_n}$ Then for all $r\in [0,\frac{r_*}{2}]$,
    $$W^2(s,t,A_n)- W^2(s,t,A_n)\leq H^2(r,s,t,A_n)-H(r,s,t,A_n).$$

    This means that the average of $H^2(r,s,t,A_n)-H(r,s,t,A_n)$ over $r\in [0,\frac{r_*}{2}]$ exceeds the value of $W^2(s,t,A_n)- W^2(s,t,A_n)$ for any $(s,t)\in\Omega^0\cap \Omega_{a_n}$. That means
    \begin{align*}
        \int_{\Omega^0\cap \Omega_{a_n}} W^2(s,t,A_n)- W^2(s,t,A_n) dm&\leq \frac{2}{r_*}\int_{\Omega^0\cap \Omega_{a_n}}\int_{0}^{\frac{r_*}{2}} H^2(r,s,t,A_n)-H(r,s,t,A_n) drdm\\
        &\leq \frac{2}{r_*}\int_{K_+} H^2(r,s,t,A_n)-H(r,s,t,A_n) drdm.
    \end{align*}

Similarly, we can show that
\begin{align*}
    \int_{\Omega^0\setminus \Omega_{a_n}} W^2(s,t,A_n)- W^2(s,t,A_n) dm&\leq \frac{2}{r_*}\int_{\Omega^0\cap \Omega_{a_n}}\int_{-\frac{r_*}{2}}^{0} H^2(r,s,t,A_n)-H(r,s,t,A_n) drdm\\
    &\leq \frac{2}{r_*}\int_{K_-} H^2(r,s,t,A_n)-H(r,s,t,A_n) drdm.
\end{align*}

We, therefore, get our desired inequality.

$$\int_{\Omega^0} W^2(s,t,A_n)- W^2(s,t,A_n) dm\leq \frac{2}{r_*}\int_{K} H^2(r,s,t,A_n)-H(r,s,t,A_n) drdm.$$

\end{proof}

Notice that $drdm$ is the Haar measure on $SL(2,\R)/\Gamma$. Since $H$ is a non-negative function function we have the following lemma.

\begin{Lemma} \label{Cor: Bound by bounding the whole space}
    With the notation as above,
$$\int_K H^2(r,s,t,A_n)-H(r,s,t,A_n) d\mu\leq \int_{SL(2,\R)/\Gamma} H^2(r,s,t,A_n)-H(r,s,t,A_n) d\mu.$$
\end{Lemma}

\begin{corollary} There exists a constant $D_1>0$ such that for $n$ large enough,
$$\int_{SL(2,\R)/\Gamma} H^2(r,s,t,A_n)-H(r,s,t,A_n) d\mu\leq  D_1\beta^2.$$
    
\end{corollary}
\begin{proof}
     %Using Corollary \ref{Cor: Bound by bounding the whole space} we have that
    % $\int_{\Omega_0}F^2(s,t)-F(s,t) dm\leq C_K\int_{SL(2,\R)/\Gamma} H^2(\theta,s,t)-H(\theta,s,t) d\mu.$ It then suffices to show that $\int_{SL(2,\R)/\Gamma} H^2(\theta,s,t)-H(\theta,s,t) d\mu \leq C'\beta^2$ for some positive constant $C'$.
Notice that for any $A\subset\R^2$, $$H(r,s,t,A)=\sum_{i=1}^mH_{i}(r,s,t,A),$$ where $H_i(r,s,t,A)=\#(\Gamma v_i \cap A)$ where $v_i$ are such that $\Lambda=\bigsqcup_{i=1}^m\Gamma v_i$. Then we get that

$$H^2(r,s,t,A)-H(r,s,t,A)=\sum_{i=1}^m H^2_i(r,s,t,A)-H_i(r,s,t,A)+ 2\sum_{1\leq j<i\leq m}H_{i}(r,s,t,A)H_{j}(r,s,t,A).$$

Theorem 1.8 from Burrin-Fairchild \cite{BurFai}, implies that for each $i=1,...,k$, there exists a constant $d_i$ such that

$$\lim_{n\rightarrow\infty} \int_{SL(2,\R)/\Gamma}H^2_i(r,s,t,A_n)-H_i(r,s,t,A_n) d\mu \leq d_i\beta^2,$$

and that there exist constants $d_{i,j}$ for $1\leq i<j\leq m$ such that

    $$\lim_{n\rightarrow\infty} \int_{SL(2,\R)/\Gamma}H_i(r,s,t,A_n)H_j(r,s,t,A_n) d\mu \leq d_{i,j}\beta^2.$$
    
   Therefore,
$$\lim_{n\rightarrow\infty}\int_{SL(2,\R)/\Gamma} H^2(r,s,t,A_n)-H(r,s,t,A_n) d\mu\leq\left(\sum_{i=1}^k d_i+2\sum_{1\leq i< j\leq k}d_{i,j}\right)\beta^2.$$
\end{proof}

Therefore, we have the following bound:

\begin{corollary}
    With the notation as above, there exists $C>0$ such that for $n$ large enough,
    $$\int_{\Omega^0}W^2(s,t,A_n)-W(s,t,A_n) dm(s,t)\leq C\beta^2 $$
\end{corollary}

\begin{proof}
    The previous three results imply that for $n$ large enough, $$\int_{\Omega^0}W^2(s,t,A_n)-W(s,t,A_n) dm(s,t)\leq \frac{4}{r_*}D_1 \beta^2$$
\end{proof}

By setting $A_n=B_{2,n}$, we get a proof for Claim \ref{Claim: Big box} and therefore the proof of Theorem \ref{Thm: Main}.

\end{document}